\documentclass[12pt,reqno,a4paper]{amsart}
\usepackage{amsmath,amssymb,amsthm,amsfonts}

\numberwithin{equation}{section}

{\theoremstyle{definition}
}
\newtheorem{theorem}{Theorem}[section]
\newtheorem{proposition}[theorem]{Proposition}

\newtheorem{corollary}[theorem]{Corollary}
\newtheorem{lemma}[theorem]{Lemma}

{\theoremstyle{definition}
{
\newtheorem{remark}[theorem]{Remark}

}}

\newcommand{\cal}{\mathcal}

\newcommand{\A}{{\cal A}}
\newcommand{\BB}{{\cal B}}

\newcommand{\DD}{{\cal D}}

\newcommand{\FF}{{\cal F}}

\newcommand{\HH}{{\cal H}}
\newcommand{\II}{{\cal I}}

\newcommand{\LL}{{\cal L}}

\newcommand{\RR}{{\cal R}}

\newcommand{\UU}{{\cal U}}
\newcommand{\VV}{{\cal V}}

\newcommand{\YY}{{\cal Y}}

\newcommand{\fU}{{\mathfrak U}}

\newcommand{\Cc}{{\mathbb{C}}}

\newcommand{\Ee}{{\mathbb{E}}}

\newcommand{\Ii}{{\mathbb{I}}}

\newcommand{\Nn}{{\mathbb{N}}}

\newcommand{\Rr}{{\mathbb{R}}}

\newcommand{\Tt}{{\mathbb{T}}}

\newcommand{\Zz}{{\mathbb{Z}}}

\def\e{\mathrm{e}}
\def\i{\mathrm{i}}

\def\diag{\operatorname{diag}}

\def\GL{\operatorname{GL}}

\def\SL{\operatorname{SL}}

\def\SLZ{\SL(d,\Zz)}
\def\SLR{\SL(d,\Rr)}

\def\veck{{\text{\boldmath$k$}}}

\def\vecu{{\text{\boldmath$u$}}}
\def\vecv{{\text{\boldmath$v$}}}

\def\vecx{{\text{\boldmath$x$}}}
\def\vecy{{\text{\boldmath$y$}}}
\def\vecz{{\text{\boldmath$z$}}}
\def\vecalf{{\text{\boldmath$\alpha$}}}
\def\vecbeta{{\text{\boldmath$\beta$}}}

\def\vecomega{{\text{\boldmath$\omega$}}}

\def\vecxi{{\text{\boldmath$\xi$}}}

\def\vecnull{{\text{\boldmath$0$}}}

\def\Im{\operatorname{Im}}

\newcommand{\norm}[1]{\left\| #1\right\|}
\newcommand{\trans} {\,^\top\!}

\newcommand{\id}  {\operatorname{Id}}
\newcommand{\im}  {\operatorname{Im}}

\newcommand{\rot}{\operatorname{Rot}}
\newcommand{\Homeo}{\operatorname{Homeo}} 
\newcommand{\Diff} {\operatorname{Diff}}  
\newcommand{\vf}   {\operatorname{Vect}} 

\newcommand{\te}[1]{\quad\text{#1}\quad}

\begin{document}

\title[Local conjugacy classes]
{Local conjugacy classes for analytic torus flows}

\author[J Lopes Dias]{Jo\~ao Lopes Dias}
\address{Departamento de Matem\'atica, ISEG\\ 
Universidade T\'ecnica de Lisboa\\
Rua do Quelhas 6, 1200-781 Lisboa, Portugal}
\email{jldias@iseg.utl.pt}

\thanks{}
\date{June 8, 2006}

\begin{abstract}
If a real-analytic flow on the multidimensional torus close enough to linear has a unique rotation vector which satisfies an arithmetical condition $\YY$, then it is analytically conjugate to linear. We show this by proving that the orbit under renormalization of a constant $\YY$ vector field attracts all nearby orbits with the same rotation vector.
\end{abstract}
 
\maketitle


\section{Introduction}

We are interested in the study of real-analytic flows $\phi^t$ on the torus $\Tt^d=\Rr^d/\Zz^d$, $d\geq2$, that are topologically conjugate to a linear flow, i.e. $h^{-1}\circ \phi^t\circ h(\vecx)=\vecx+\vecomega t \bmod1$, $t\geq0$, for some homeomorphism $h$ of $\Tt^d$ and $\vecomega\in\Rr^d$. Our aim, as part of a program, is to find sufficient conditions for which the conjugacy $h$ is guaranteed to be real-analytic.

The general case $d\geq2$ differs significantly from the lower dimension situation $d=2$ where there is an invariant of motion, the asymptotic direction of the flow, whose slope is called rotation number.
The classical Denjoy's theorem \cite{Denjoy} asserts that for sufficiently smooth flows, irrational rotation numbers describe completely topological conjugacy classes.
This can not be generalized to higher dimensions due to the large variety of possible dynamical behaviours; often there is more than one asymptotic direction (rotation vectors).
In this paper we look at analytic flows with a unique rotation vector for all orbits, showing that these have similar properties to the ones of $d=2$, at least for 'typical' rotation vectors.

The two dimensional theory was further developed by Arnol'd \cite{Arnold2}, Herman \cite{Herman1979} and Yoccoz \cite{Yoccoz} in the 1960's, 1970's and 1980's, respectively. Mainly working on the discrete-time counterpart, circle diffeomorphisms as return maps to flow transversals, they showed that the conjugacy classes correspond to higher regularity (see also \cite{Katznelson,Katznelson2}). However, this will depend on the arithmetical properties of the rotation number due to their crucial role in solving small divisor problems. In particular, in the 1990's Yoccoz proved that for analytic diffeomorphisms close to a rotation, the conjugacy is analytic if the rotation number is of Brjuno type \cite{Yoccoz2}; and this arithmetical condition is optimal. Moreover, the closeness to rotation hypothesis can be dropped if restricting to Yoccoz's $\HH$ set of rotation numbers \cite{Yoccoz4}. Again, this condition is optimal.

The above cited results can be seen as proofs of differential rigidity within topological conjugacy classes of smooth systems. This has been observed to be a property common to different problems as in e.g. \cite{Khanin03,Faria,Faria2,Faria3}, while renormalization being the tool often used. However, in those works the renormalizations are based on the continued fraction expansion of irrational numbers. Here we are able to work on higher dimensions by making use of the multidimensional continued fractions algorithm introduced in \cite{jld5} (see also \cite{jld9,jld8,Koch-Kocic06}). This roughly corresponds to a flow on the homogeneous space $\SL(d,\Zz)\setminus\SL(d,\Rr)$ that, following Lagarias ideas \cite{Lagarias94}, provides a strongly convergent continued fractions expansion for all vectors (see section \ref{section:Multidimensional continued fractions} below).

Given $\vecomega\in\Rr^d-\{0\}$, define the linear torus flow 
\begin{equation}
R_\vecomega^t\colon  \Tt^d  \to \Tt^d,
\qquad
\vecx \mapsto  \vecx+\vecomega t \bmod1,
\end{equation}
with $t\geq0$.
The main result in this paper is the following.

\begin{theorem}\label{main thm I}
Let $\vecomega\in\YY\subset\Rr^d$.
If a $C^\omega$ flow on $\Tt^d$ has a unique rotation vector $\vecomega$ and close enough to linear, then it is $C^\omega$-conjugate to $R_\vecomega^t$.
\end{theorem}

Since the rotation vector is invariant under homeomorphisms (see section \ref{section:Rotation vectors}), this result yields  an immediate consequence:

\begin{corollary}
Let $\vecomega\in\YY\subset\Rr^d$.
If a $C^\omega$ flow on $\Tt^d$ is $C^0$-conjugate to $R_\vecomega^t$ and it is close enough to linear, then the conjugacy is in fact $C^\omega$.
\end{corollary}

Therefore, the topological and analytic conjugacy classes of $R_\vecomega^t$ are locally the same for $\vecomega\in\YY$ and coincide with the set of close-to-linear flows with a rotation vector $\omega$.
Note that by allowing time-reparametrizations and recalling that $R_{\lambda\vecomega}^t=R_\vecomega^{\lambda t}$, $\lambda\not=0$, we obtain larger conjugacy classes.

In Proposition \ref{prop DC in YY} we will see that the set $\YY$ contains all diophantine vectors, having full Lebesgue measure.
In dimension two $\YY$ corresponds to the set of vectors whose slope is a Brjuno number.

We highlight the fact that no condition such as volume-preservation is required for the main theorem to hold. It is known \cite{Herman1975,Herman1979,Moser2} that volume-preserving $C^\omega$-flows with a diophantine rotation vector $\vecomega$ and close to $R_\vecomega^t$ are $C^\omega$-conjugated to $R_\vecomega^t$. That is a consequence of the existence of a parameter $\lambda$ which makes a vector field $X+\lambda$ (not necessarily divergence free) conjugated to a translation, where $X$ is close to $\vecomega$ (see also \cite{jld5}). With $R_\vecomega^t$ ergodic, if the volume is preserved and the rotation vector is $\vecomega$, then $\lambda$ has to vanish (Proposition 2.6.1 \cite{Herman1979} pp.180).

We believe that further extensions of Theorem \ref{main thm I}, namely for $C^r$ vector fields, should be accessible by the present method. Two open and outstanding problems are the corresponding global result and the determination of an optimal condition on $\vecomega$ for which the theorem holds.

The proof of Theorem \ref{main thm I} is a consequence of the convergence under renormalization of vector fields in some small ball around $\vecomega$. The renormalization operator is basically a coordinate change and time rescaling related to the continued fractions of $\vecomega$ (cf. \cite{Koch,jld5}). Convergence is guaranteed if the rotation vector $\vecomega\in\YY$. The fact that we have a unique rotation vector permits us to control the distance between $\vecomega$ and the constant Fourier mode of the perturbed vector fields. Our scheme therefore contracts a ball in its domain towards the orbit under renormalization of the constant vector field $\vecomega$. The differentiable conjugacy then follows.

In section \ref{section:Multidimensional continued fractions} we review the multidimensional continued fractions scheme introduced in \cite{jld5}, and define the class of $\YY$ vectors.
In sections \ref{section:Rotation vectors} and \ref{section:Preliminaries} we present the renormalization building blocks, which will be put together in section \ref{section:Renormalization} in order to determine sufficient conditions for the existence of infinitely renormalizable vector fields.
In section \ref{section:Differentiable rigidity} we construct the analytic conjugacy for vector fields which are attracted under renormalization to the orbit of the constant system. The proof of Theorem \ref{main thm I} is concluded in section \ref{section:proof local rig thm}.

 For the following we set the notations $\Nn=\{1,2,\dots\}$ for the positive integers and $\Nn_0=\Nn\cup\{0\}$ for the non-negative integers.
Moreover, $A\ll B$ stands for the existence of a constant $C>0$ such that $A\leq CB$.


\section{Multidimensional continued fractions}
\label{section:Multidimensional continued fractions}

 In this section we present the multidimensional continued fractions algorithm introduced in \cite{jld5}. In addition, we define the class of vectors $\YY$ from the properties of the continued fractions expansion in an appropriate way to be used later by the renormalization scheme.

\subsection{Flow on homogeneous space}

Denote by $G=\SLR$, $\Gamma=\SLZ$ and take a fundamental domain $\FF\subset G$ of the homogeneous space $\Gamma\backslash G$ (the space of $d$-dimensional non-degenerate unimodular lattices).
On $\FF$ consider the flow:
\begin{equation}\label{def geod flow}
\Phi^t\colon \FF \to\FF,
\quad
M\mapsto P(t) M E^t,
\end{equation}
where 
$$
E^t=\diag(\e^{-t}, \dots, \e^{-t}, \e^{(d-1)t}) \in G
$$
and $P(t)$ is the unique family in $\Gamma$ that keeps $\Phi^tM$ in $\FF$ for every $t\geq0$.

 For the remaining of this paper we fix $\vecomega=(\vecalf,1)\in\Rr^d$. We are then interested in the orbit under $\Phi^t$ of the matrix 
\begin{equation}
M_\vecomega=
\begin{pmatrix}
I & \vecalf \\
\vecnull & 1
\end{pmatrix}.
\end{equation}

\subsection{Growth of the flow}

Let the function $\delta\colon \Gamma\backslash G\to\Rr^+$ measuring the shortest vector in the lattice $M$ be
\begin{equation}
\delta(M)=\inf_{\veck\in\Zz^d-\{0\}}\|\trans \veck M\|,
\end{equation}
where $\|\cdot\|$ stands for the $\ell_1$-norm (in the following we will make use of the corresponding matrix norm taken as the usual operator norm).
Notice that $\delta(\Phi^t M_\vecomega)=\delta(M_\vecomega E^t)$.

\begin{proposition}[\cite{jld5}]\label{bnds Mn}
There exist $C_1,C_2>0$ such that for all $t\geq0$
\begin{equation}
\|\Phi^t M_\vecomega\| \leq \frac{C_1}{\delta(\Phi^t M_\vecomega)^{d-1}}
\te{and}
\|(\Phi^t M_\vecomega)^{-1}\| \leq \frac{C_2}{\delta(\Phi^t M_\vecomega)}.
\end{equation}
\end{proposition}

\subsection{Stopping times}
\label{subsection:Stopping times}

Consider a sequence of times, called {\em stopping times},
\begin{equation}\label{time seq}
t_0=0 < t_1 < t_2 < \dots \to +\infty
\end{equation}
such that the matrices $P(t)$ in \eqref{def geod flow} satisfy
\begin{equation}
P_n:=P(t_n)\not= P(t_{n-1}),
\end{equation}
with $n\in\Nn$. We also set $P_0=P(t_0)=I$.
The sequence of matrices $P_n\in\SL(d,\Zz)$ are the rational approximates of $\vecomega$, called the {\em multidimensional continued fractions expansion}.
In addition we define the transfer matrices
\begin{equation}\label{defn Tn}
T_n=P_n P_{n-1}^{-1},
\quad
n\in\Nn,
\te{and}
T_0=I.
\end{equation}

The flow of $M_\vecomega$ taken at the time sequence is thus the sequence of matrices
\begin{equation}\label{defn Mn}
M_n=\Phi^{t_n}M_\vecomega=P_n M_\vecomega E^{t_n}.
\end{equation}
Using some properties of the flow, the above can be decomposed (see \cite{jld5}) into
\begin{equation}\label{def Mn 2}
M_n=
\begin{pmatrix}
I & \vecalf_n \\
\vecnull & 1
\end{pmatrix}
\begin{pmatrix}
\Delta_n & \vecnull \\
\trans\vecbeta_n & \gamma_n
\end{pmatrix}
\end{equation}
with $\gamma_n$ being the $d$-th component of the vector $\e^{(d-1)t_n} P_n\vecomega$.

Define $\vecomega_n=(\vecalf_n,1)$, $\vecomega_0=\vecomega$ and, for $n\in\Nn$,
\begin{equation} \label{eq def omega n}
\vecomega_n 
= \gamma_n^{-1} M_n
\left(\begin{smallmatrix}0\\ \vdots \\ 0 \\1\end{smallmatrix}\right)
= \lambda_nP_n\vecomega
= \eta_nT_n\vecomega_{n-1},
\end{equation} 
where 
\begin{equation} \label{eq def eta n lambda n}
\lambda_{n} =
\frac 1{\gamma_n} \e^{(d-1)t_{n}}
\quad\text{and}\quad
\eta_n = 
\frac{\lambda_n}{\lambda_{n-1}}.
\end{equation}

We remark that when $d=2$, there exists a sequence of stopping times (called Hermitte critical times) that gives an accelerated version of the standard continued fractions of a number $\alpha$ \cite{jld5,Lagarias94}.

\subsection{Resonance widths}

We call {\em resonance widths} to the terms of any decreasing sequence $\sigma\colon\Nn_0\to\Rr^+$.

\subsection{Resonance cone}

Given resonance widths $\sigma\colon\Nn_0\to\Rr^+$, define the resonant cones to be 
\begin{equation}\label{def I+}
I_n^+ =\{ \veck \in\Zz^d \colon 
|\veck\cdot\vecomega_n |\leq \sigma_n \|\veck\| \}.
\end{equation}
In addition, let
\begin{equation}\label{def An}
A_n
=\sup_{\veck\in I_n^+-\{0\}}
\frac{\|\trans T_{n+1}^{-1}\veck\|}{\|\veck\|}.
\end{equation}

\begin{proposition}\label{lemma resonance cone}
There is $c>0$ such that for any $n\in\Nn_0$
\begin{equation}\label{formula An}
A_n
\leq
c\, \e^{-\delta t_{n+1}}
\frac{ \sigma_{n}\e^{d\delta t_{n+1}}  + 1 }
{\delta(M_n)^{d-1} \delta(M_{n+1})},
\end{equation}
where $\delta t_{n+1}=t_{n+1}-t_{n}$.
\end{proposition}

\begin{proof}
Take $S^\perp_n$ to be the hyperspace orthogonal to $\vecomega_n$. By \eqref{defn Tn} and \eqref{defn Mn}, $T_n=M_nE^{-\delta t_n}M_{n-1}^{-1}$ and $\trans T_{n+1}^{-1}=\trans M_{n+1}^{-1}E^{\delta t_{n+1}}\trans M_n$. So, for $\vecxi\in S_n^\perp$, \eqref{def Mn 2} gives
$$
\trans T_{n+1}^{-1}\vecxi = 
\e^{-\delta t_{n+1}}\trans M_{n+1}^{-1}\begin{pmatrix}\trans A_n\vecxi'\\0\end{pmatrix}
=
\e^{-\delta t_{n+1}}\trans M_{n+1}^{-1}\trans M_n\vecxi,
$$
where $\vecxi'\in\Rr^{d-1}$ comprises the first $d-1$ components of $\vecxi$.

Now, write $\veck\in I_n^+-\{0\}$ as $\veck=\veck_1+\veck_2$ where
$$
\veck_1=\frac{\veck\cdot\vecomega_n}{\vecomega_n\cdot\vecomega_n}\vecomega_n
\te{and}
\veck_2\in S_n^\perp.
$$
Hence,
\begin{equation}
\begin{split}
\|\trans T_{n+1}^{-1}\veck\| 
&\leq 
\|\trans T_{n+1}^{-1}\veck_1\| + 
\|\trans T_{n+1}^{-1}\veck_2\| \\
&\leq
\sigma_n\|\trans T_{n+1}^{-1}\| \, \|\veck\| +
\e^{-\delta t_{n+1}} \|\trans M_{n+1}^{-1}\| \, \|\trans M_n\| \, \|\veck\| \\
&\leq
(\sigma_n\|E^{\delta t_{n+1}}\|  +
\e^{-\delta t_{n+1}}) \|\trans M_{n+1}^{-1}\| \, \|\trans M_n\| \, \|\veck\|
\end{split}
\end{equation}
which can be estimated using Proposition \ref{bnds Mn}.
\end{proof}


\subsection{Class of diophantine vectors}
\label{subsection:Class of diophantine vectors}

A vector $\vecomega\in\Rr^d$ is diophantine with exponent $\beta\geq0$ if there is a constant $C>0$ such that 
$$
|\vecomega\cdot\veck| > \frac{C}{\|\veck\|^{d-1+\beta}}.
$$
It is a well known fact that the sets $DC(\beta)$ of diophantine vectors with exponent $\beta>0$ are of full Lebesgue measure \cite{Cassels57}.
On the other hand, the set $DC(0)$ has zero Lebesgue measure.
A vector is said to be diophantine if it belongs to $DC=\cup_{\beta\geq0}DC(\beta)$.

\begin{proposition}\label{dioph bdd}
Let $\beta\geq0$. Then, $\vecomega\in DC(\beta)$ iff there is $C'>0$ such that
$$
\delta(\Phi^t M_\vecomega)> C'\e^{-\theta t},
\quad
t\geq0,
$$
with $\theta=\beta/(d+\beta)$.
\end{proposition}

\begin{proof}
Notice first that $\trans \veck M_\vecomega E^t=(\e^{-t}\widehat\veck, \e^{(d-1)t}(\veck\cdot\vecomega))$, where $\veck=(\widehat\veck,k_d)\in\Zz^d-\{0\}$.
Now, we have that 
\begin{equation}\label{eq dioph bdd 1}
C'< \e^{\theta t}\inf_{\veck\not=0}\|\trans k M_\vecomega E^t\| \leq \max\{\e^{-(1-\theta)t}\|\veck\|,\e^{(d-1+\theta)t}|\veck\cdot\vecomega|\}.
\end{equation}
Let
$$
t=t(\veck)=
\begin{cases}
\frac1d\log\frac{\|\veck\|}{|\veck\cdot\vecomega|}, 
& \|\veck\|\geq|\veck\cdot\vecomega| \\
0, & \|\veck\| \leq |\veck\cdot\vecomega|.
\end{cases}
$$
Using this $t$ in \eqref{eq dioph bdd 1}, the second case trivially means that $|\veck\cdot\vecomega|> C' > C \|\veck\|^{-(d-1+\beta)}$ for some constant $C>0$. For the first case,
$$
C'< \|\veck\|^{(d-1+\theta)/d}|\veck\cdot\vecomega|^{1-(d-1+\theta)/d}.
$$
So, $|\veck\cdot\vecomega|>C'\|\veck\|^{d-1+\beta}$ and $\vecomega$ is diophantine with exponent $\beta$.

The converse is proved in \cite{jld5}-Lemma 2.3.
\end{proof}

\begin{proposition}\label{proposition norm estimates}
If $\vecomega\in DC(\beta)$, $\beta\geq0$, 
there are constants $c_1,c_2,c_3,c_4,c_5,c_6,c_7>0$ 
such that, for all stopping-time sequence $t\colon\Nn_0\to\Rr$,
\begin{eqnarray}
\label{M1}
\| M_n \| &\leq& c_1\exp[(d-1)\theta t_{n}] , \\
\label{M2}
\| M_n^{-1} \|  &\leq& c_2 \exp(\theta t_n), \\
\label{P1}
\| P_n \|  &\leq& c_3\exp[(d\,\theta+1-\theta) t_{n}], \\
\label{P2}
\| P_n^{-1} \|  &\leq& c_4 \exp[(d-1+\theta) t_n], \\
\label{T1}
\| T_n \|&\leq& c_5\exp[(1-\theta)\delta t_n+d\,\theta\, t_n], \\
\label{T2}
\| T_n^{-1} \| 
&\leq& c_6 \exp[(d-1)(1-\theta)\delta t_n + d\,\theta\,t_n],
\end{eqnarray}
and
\begin{equation}\label{gamma1}
c_7
\exp\bigg[-\theta \bigg(\frac{d^2}{1-\theta}-(d-1)\bigg) t_n\bigg] 
\leq 
|\gamma_n|
\leq 
c_1
\exp[(d-1)\theta t_{n}],
\end{equation}
where $\delta t_n=t_n-t_{n-1}$ and $\theta=\beta/(d+\beta)$.
\end{proposition}

\begin{proof}
These estimates are a direct consequence of Proposition \ref{dioph bdd} applied to Proposition \ref{bnds Mn} (cf. \cite{jld5}).
\end{proof}

\begin{proposition}\label{prop hyp cone for dioph}
If $\vecomega\in DC(\beta)$, $\beta\geq0$, then there is $c>0$ such that for any $n\in\Nn_0$,
\begin{equation}
A_n
\leq
c\, \e^{-(1-\theta)\delta t_{n+1}+d\theta t_{n}}
\left(
\sigma_{n} \e^{d\delta t_{n+1}}+1\right).
\end{equation}
\end{proposition}

\begin{proof}
It follows immediately by applying Proposition \ref{dioph bdd} to Proposition \ref{lemma resonance cone}.
\end{proof}

\subsection{Class of $\YY$-vectors}

In this paper we will be dealing with a class of vectors which satisfies some arithmetical properties more general than the diophantine ones (cf. section \ref{subsection:Class of diophantine vectors}).

A vector $\vecomega\in\Rr^d$ belongs to $\YY$ if there exists sequences of resonance widths $\sigma\colon\Nn_0\to\Rr^+$ and of stopping-times $t\colon\Nn_0\to\Rr$ such that
\begin{equation}\label{defn YY vectors}
\sum_{n\geq0}A_0\dots A_n 
\log\left(|\eta_{n+1}|\,\|T_{n+1}\| \, \frac{\sigma_n\|\vecomega_{n+1}\|}{\sigma_{n+1}}\right)
<\infty
\end{equation}
and
\begin{equation}\label{cdn on sigma n for YY}
\lim_{n\to+\infty}2^{2n}\sigma_n\|P_n^{-1}\|\,\|T_0\|^2\dots\|T_n\|^2=0.
\end{equation}
The conditions above should be seen as lower and upper bounds on $\sigma_n$.
It follows immediately that $\vecomega_N\in\YY$ for any fixed $N\in\Nn$, using shifted sequences $\sigma_{N+n}$ and $t_{N+n}$.

The class of $\YY$ vectors contains the full probability set of diophantine vectors as it is proved below.

\begin{proposition}\label{prop DC in YY}
$DC\subset\YY$.
\end{proposition}

\begin{proof}
Let $0<\xi<1-\theta$, the stopping-time sequence given recursively by $t_{n+1}=\e^{(1-\theta-\xi)t_n}$, $n\in\Nn$, $t_0=0$, and the sequences of resonance widths given by $\sigma_n= \e^{-d\delta t_{n+1}}$ for $n$ large enough.

For $\vecomega\in DC(\beta)$, $\beta\geq0$, using Propositions \ref{proposition norm estimates} and \ref{prop hyp cone for dioph}, the terms in the series in \eqref{defn YY vectors} can be bounded from above by
$$
C^nt_{n+2}\e^{-(1-\theta)t_{n+1}+d\theta\sum_{i=0}^nt_i},
$$
for some constant $C>0$. By our choice of the stopping-times, the above can be estimated from above by $C'\e^{(1-\theta-\xi)t_{n+1}-(1-\theta-\xi')t_{n+1}}$ with $C'>0$ and $0<\xi'<\xi$. Therefore, the series in \eqref{defn YY vectors} converges.

On the other hand, the expression in the limit of \eqref{cdn on sigma n for YY} can be bounded by
$$
2^{2n}\e^{-d\delta t_{n+1}+(d+1-\theta)t_n+2\theta d\sum_{i=0}^nt_i}
$$
which goes to zero as $n\to+\infty$ by our present choice of stopping times.

This completes the proof that $\vecomega\in\YY$.
\end{proof}

\begin{remark}
If we restrict to $d=2$ it is natural to use the standard continued fractions expansion. In this case it can be shown that $\YY$ corresponds to the Brjuno vectors, i.e. vectors whose slope is a Brjuno number.
\end{remark}


\section{Rotation vectors}
\label{section:Rotation vectors}

We define the rotation vector of a flow $\phi^t$ at each $\vecx\in\Tt^d$ to be the asymptotic direction of the corresponding orbit of the lift $\Phi^t(\vecx)$ to the universal cover:
\begin{equation}
\rot(\phi)(\vecx)=\lim_{t\to\infty}\frac{\Phi^t(\vecx)-\vecx}{t},
\end{equation}
if the limit exists.

\begin{remark}
If the rotation vector exists at $\vecx$ for a flow $\phi^t$ generated by a vector field $X$ on $\Tt^d$ (i.e. $\frac d{dt}\phi^t=X\circ\phi^t$), it is the time average of the vector field along the orbit:
\begin{equation}
\rot(\phi)(\vecx)= \lim_{t\to\infty} \frac1t \int_0^t X\circ\phi^s(\vecx)ds.
\end{equation}
\end{remark}

When the rotation vector exists for all $\vecx\in\Tt^d$,
the rotation set of $\phi$ is
\begin{equation}
\rot(\phi)=\{\rot(\phi)(\vecx)\colon \vecx\in\Tt^d\}.
\end{equation}

Throughout this text we denote by $\Homeo(M)$ and $\Diff^r(M)$, $r\in\Nn\cup\{\infty,\omega\}$, the set os homeomorphisms and $C^r$-diffeomorphisms on $M$. Moreover, we add a subscript $0$ to distinguish the case of isotopic to the identity maps. Finally, $\vf^r(M)$ stands for the set of $C^r$-vector fields on $M$.

\begin{lemma}\label{properties rot}
Let $h\in\Homeo_0(\Tt^d)$, $\lambda\not=0$ and $T\in\GL(d,\Zz)$.
If $\rot(\phi)\not=\emptyset$, then
\begin{equation}
\rot(h^{-1}\circ\phi\circ h)=\rot(\phi)
\quad\text{and}\quad
\rot(T^{-1}\circ\phi^{\lambda\cdot}\circ T)= \lambda T^{-1}\rot(\phi).
\end{equation}
\end{lemma}

\begin{proof}
Writing $h^{-1}=\id+\varphi$ with $\varphi$ a $\Zz^d$-periodic function, we get
\begin{equation}
\rot(h^{-1}\circ\phi\circ h)(\vecx)=
\lim_{t\to\infty} \frac1t[\Phi^t\circ h(\vecx)+\varphi\circ\phi^t\circ h(\vecx)-h(\vecx)+h(\vecx)-\vecx].
\end{equation}
The fact that $\varphi$ is bounded and that there is a rotation vector for all points in $\Tt^d$,
yields $\rot(h^{-1}\circ\phi\circ h)(\vecx)=\rot(\phi)(h(\vecx))$.
Thus the first assertion.

The second claim follows from
\begin{equation}
\begin{split}
\rot(T^{-1}\circ\phi^{\lambda\cdot}\circ T)(\vecx) 
&= 
\lim_{t\to\infty} \frac1t(T^{-1}\circ\Phi^{\lambda t}\circ T\vecx-\vecx)
\\ 
&=
\lim_{t\to\infty} \frac1{\lambda t}\lambda T^{-1}(\Phi^{\lambda t}\circ T\vecx-T\vecx)
\\
&=
\lambda T^{-1} \rot(\phi)(T\vecx).
\end{split}
\end{equation}
\end{proof}

\begin{proposition}\label{prop rot control EX}
Let $\phi^t$ be the flow generated by $X\in\vf^0(\Tt^d)$ and $\vecomega\in\Rr^d$. If $\rot \phi=\{\vecomega\}$, then 
\begin{equation}
\|\Ee X-\vecomega\| \leq d \|X-\Ee X\|_{C^0},
\end{equation}
where $\Ee X=\int_{\Tt^d}X\,dm$ and $m$ denotes the Lebesgue measure on $\Tt^d$.
\end{proposition}

\begin{proof}
We first show that for each $1\leq i\leq d$ there is $\vecy^{(i)}\in\Tt^d$ such that $X_i(\vecy^{(i)})=\vecomega_i$. (We represent the $i$th coordinate of vectors by the subscript $i$).
This follows from the fact that $\rot \phi(\vecx)-\vecomega=0$.
I.e. for any $\vecx\in\Tt^d$ and $1\leq i\leq d$,
$$
\lim_{t\to+\infty}\frac1t\int_0^t[X_i\circ\phi_s(\vecx)-\omega_i]ds=0.
$$
The continuous function $\psi(\vecx,t)=X_i\circ\phi_t(\vecx)-\omega_i$ on $\Tt^d\times\Rr^+$ has $m=\min \psi$ and $M=\max \psi$ because $X$ is continuous on a compact set. So, for any $(\vecx,t)$ we have $m\leq \frac1t\int_0^t\psi(\vecx,s)ds\leq M$.
Taking the limit, $m\leq0\leq M$ and we can find a zero of $\psi$, hence of $X_i-\omega_i$.

Now, using the above points $\vecy^{(i)}$,
\begin{equation*}
\begin{split}
\|\Ee X-\vecomega\| & = 
\sum_{i=1}^d|\Ee X_i-\omega_i| = 
\sum_{i=1}^d|X_i(\vecy^{(i)})-\Ee X_i| \\
& \leq
\sum_{i=1}^d \max_{\vecx}|X(\vecx)-\Ee X_i|
\leq
d\,\|X-\Ee X\|_{C^0}.
\end{split}
\end{equation*}
\end{proof}

\begin{remark}
We will be interested in vector fields generating flows that possess the same rotation vector for all orbits.
Hence, for a vector field $X$ we will write $\rot X$ to mean the unique rotation vector associated to the flow generated by $X$.
\end{remark}


\section{Preliminaries}
\label{section:Preliminaries}

\subsection{Definitions}

The transformation of $X\in\vf(M)$ on a manifold $M$ by $\psi\in\Diff(M)$ is given by the {\it pull-back} of $X$ under $\psi$:
$$
\psi^*X=(D\psi)^{-1}X\circ \psi.
$$
As $T\Tt^d\simeq\Tt^d\times\Rr^d$, we identify the
set of vector fields on $\vf(\Tt^d)$ with the set of functions $C(\Tt^d,\Rr^d)$, that can be regarded as $\Zz^d$-periodic maps of $\Rr^d$ by lifting to the
universal cover.
We will make use of the analyticity to extend to the complex domain,
so we will deal with complex analytic functions.



\subsection{Space of vector fields}

Let $\rho>0$ and the domain
\begin{equation}
D_{\rho} =\{\vecx\in\Cc^{d} \colon \|\Im\vecx\| <\frac{\rho}{2\pi}\},
\end{equation}
for the norm $\|{\vecu}\|=\sum_i|u_i|$ on $\Cc^d$.

Take complex analytic functions $f\colon D_\rho \to\Cc^d$ that are
$\Zz^d$-periodic and on the form of the Fourier series 
\begin{equation}\label{F series vf}
f(\vecx)=\sum\limits_{\veck\in\Zz^d}f_\veck\e^{2\pi \i \veck\cdot \vecx}
\end{equation}
with $f_\veck\in\Cc^d$.
The Banach spaces $\A_\rho$ and $\A_\rho'$ are the subspaces of such functions with the respective finite norms 
\begin{equation}
\|f\|_\rho =
\sum\limits_{\veck\in\Zz^d} \|f_\veck\| \, \e^{\rho\|\veck\|}
\te{and}
\|f\|_\rho'=
\sum\limits_{\veck\in\Zz^d} \left(1+2\pi\|\veck\|\right)\|f_\veck\| \,\e^{\rho\|\veck\|}.
\end{equation}
%
%
Consider also the norm $\|f\|_{C^0}=\max_{\vecx\in\Tt^d}\|f(\vecx)\|$.

Some of the properties of the above spaces are of easy verification.
For instance, given any $f,g\in\A_\rho'$ we have:
\begin{itemize}
\item
$\|f(\vecx)\|\leq\|f\|_{C^0}\leq \|f\|_{\rho}\leq\|f\|'_{\rho}$ where $\vecx\in D_{\rho}$,
\item
$\|f\|_{\rho-\delta}\leq\|f\|_{\rho}$ with $0\leq\delta\leq\rho$,
\item
$\|Df\|_\rho\leq \delta^{-1} \|f\|_{\rho+\delta}$ with $\rho,\delta>0$ (Cauchy's estimate).
\end{itemize}

Write the constant Fourier mode of $f\in\A_\rho$ through the projection
\begin{equation}
\Ee f=\int_{\Tt^d}f\,dm=f_0 \in\Cc^d.
\end{equation}

We will only be interested in the above vector fields which are fixed-point-free, i.e. $f(\vecx)\not=0$ for all $\vecx\in D_\rho$.

\subsection{Far from resonance modes}

Given $\sigma>0$, we call {\em far from resonance} modes with respect to $\vecv\in\Rr^d$ to the Fourier modes with indices in 
\begin{equation}\label{def I-}
I_\sigma^- =\left\{ \veck\in \Zz^d \colon 
|\vecv\cdot \veck| > \sigma \|\veck\| 
\right\}.
\end{equation}
The {\em resonant modes} are the ones in $I_\sigma^+=\Zz^d- I_\sigma^-$ which correspond to the resonant cones defined in \eqref{def I+}.
We also have the projections $\Ii_\sigma^+$ and $\Ii_\sigma^-$ over the spaces of vector fields by restricting the modes to $I_\sigma^+$ and $I_\sigma^-$, respectively.
The identity operator is $\Ii=\Ii_\sigma^++\Ii_\sigma^-$.

\subsection{Uniformization}

We call uniformization of a vector field to the action of a diffeomorphism that produces a new vector field with only resonant modes.

Given $\rho,\varepsilon,\nu>0$, denote by $\VV_\varepsilon$ the open ball in $\A'_{\rho+\nu}$ centred at $\vecv\not=0$ with radius $\varepsilon$.

\begin{theorem}[\cite{jld,jld5}]\label{thm unif}
Let $\sigma<\|\vecv\|$ and
\begin{equation}\label{formula epsilon}
\varepsilon=
\frac{\sigma}{42}
\min\left\{
\frac{\nu}{4\pi},
\frac{\sigma}{72\|\vecv\|}
\right\}.
\end{equation}
There exists a 1-parameter smooth family of maps $\fU_t\colon  \VV_{\varepsilon} \to \A'_{\rho}$ and $\UU_t\colon \VV_{\varepsilon} \to  \Ii_\sigma^+\A_{\rho}\oplus(1-t)\Ii_\sigma^-\A'_{\rho+\nu}$ given by $\UU_t(X)=\fU_t(X)^*X$ such that
\begin{equation}\label{equation homotopy method}
\Ii_\sigma^-\UU_t(X)=(1-t)\,\Ii_\sigma^-X,
\qquad
t\in[0,1],
\end{equation}
and
\begin{equation}\label{estimate U around X0}
\begin{split}
\|\fU_t(X)-\id\|'_{\rho}
\leq & 
\frac{42t}{\sigma} \|\Ii_\sigma^-X\|_{\rho}\\
\|\UU_t(X)-\vecv\|_{\rho}
\leq &
(3-t)
\|X-\vecv\|'_{\rho+\nu}.
\end{split}
\end{equation}
Moreover, if $X$ is real-analytic, $\fU_t(X)(\Rr^{2d})\subset\Rr^{2d}$.
\end{theorem}

Let the translation $R_\vecz$ on $\Cc^d$ for each
$\vecz\in\Cc^d$ be 
\begin{equation}\label{translation Rz}
R_\vecz\colon \vecx\mapsto \vecx+\vecz.
\end{equation}

\begin{lemma}\label{theorem 2 elim terms}
In the conditions of Theorem {\rm \ref{thm unif}},
if $\vecx\in \Rr^d$ and $X\in \VV_{\varepsilon}$, then
\begin{equation}\label{relation tilde U and U}
\fU_t(X\circ R_\vecx)=R_\vecx^{-1}\circ \fU_t(X)\circ R_\vecx
\end{equation}
on $\DD_{\rho}$.
\end{lemma}

\begin{proof}
Notice that $R_\vecx (D_{\rho})=D_{\rho}$.
If $U_t=\fU_t(X)$ is a solution of the homotopy equation \eqref{equation homotopy method} on $D_{\rho}$, then $\tilde U_t=R_\vecx^{-1}\circ \fU_t(X)\circ R_\vecx$ solves the same equation for $\tilde X=X\circ R_\vecx$, i.e. $\Ii_\sigma^-\tilde X\circ \tilde U_t=(1-t)\Ii_\sigma^-\tilde X$, on $D_{\rho}$.
\end{proof}

\subsection{Rescaling}

The fundamental step of the renormalization is a linear transformation of the
domain of definition of our vector fields.
This is done by a change of basis using the multidimensional continued fractions matrices $T_n\in\SL(d,\Zz)$ of a vector $\vecomega\in\Rr^d$. 
The normalizing scalars $\eta_n$ in \eqref{eq def eta n lambda n} are used for a linear time rescaling.

Consider $X\in\A_{\rho}$.
We are interested in the following coordinate and time linear changes:
\begin{equation}\label{coord x and time change}
L_n\colon\vecx\mapsto T_n^{-1}\vecx,
\qquad
t\mapsto \eta_n t.
\end{equation}
(Negative time rescaling means inverting the direction of time.)
These determine a new vector field as the image of the map
$$
X\mapsto\LL_n(X)=\eta_n L_n^*X = \eta_n T_n X\circ T_n^{-1}.
$$
Thus the following relation holds:
\begin{equation}\label{commutation property Rz and Ln}
L_n^*R_\vecz^*=R_{T_n\vecz}^*L_n^*,
\qquad
\vecz\in\Cc^d.
\end{equation}

Recall the definition of $A_n$ given by \eqref{def An} for some choice of $\sigma_n$. Notice that, according to \eqref{def I+}, $I^\pm_n=I^\pm_{\sigma_n}$ the far from resonance modes with respect to $\vecomega_n$, and we write $\Ii^\pm_n=\Ii^\pm_{\sigma_n}$.

\begin{lemma}\label{proposition TT}
If $\delta>0$ and
\begin{equation}\label{hyp on rho n' for vf}
\rho_n'\leq\frac{\rho_{n-1}}{A_{n-1}}-\delta,
\end{equation}
then $\widetilde\LL_n$ given by $\LL_n$ restricted to $(\Ii^+_{n-1}-\Ee)\A_{\rho_{n-1}}$ into $(\Ii-\Ee)\A'_{\rho_n'}$ is bounded 
with
\begin{equation}\label{bdd tilde LL n}
\|\widetilde\LL_n\|\leq
|\eta_{n}|\, \|T_n \|\,
\left(1+\frac{2\pi}{\delta}\right).
\end{equation}
\end{lemma}

\begin{proof}
Let $f\in(\Ii^+_{n-1}-\Ee)\A_{\rho_{n-1}}$.
Then,
\begin{equation}
\|f\circ L_n\|'_{\rho_n'}
\leq
\sum_{\veck\in I^+_{n-1}-\{0\}}
\left(1+2\pi\|\trans T_n^{-1}\,\veck\|\right)
\|f_\veck\|
\e^{(\rho_n'-\delta+\delta)\|\trans T_n^{-1} \veck\|}.
\end{equation}
By using the relation $\xi \e^{-\delta\,\xi}\leq \delta^{-1}$ with $\xi\geq0$, and \eqref{def An}, we get
\begin{equation}
\|f\circ L_n\|'_{\rho_n'}
\leq
\left(1+\frac{2\pi}{\delta}\right)
\sum_{I^+_{n-1}-\{0\}} 
\|f_\veck\|
\e^{A_{n-1}(\rho_n'+\delta)\|\veck\|}
\leq 
\left(1+\frac{2\pi}{\delta}\right)
\|f\|_{\rho_{n-1}}.
\end{equation}
Finally, $\|\widetilde\LL_n f\|'_{\rho_n'}\leq |\eta_{n}|\,\|T_n \| \,\|f\circ L_n\|'_{\rho_n'}$.

\end{proof}


\subsection{Analyticity strip cut-off}
\label{subsection:Analyticity strip cut-off}

Consider the operator $\II\colon \A_{\rho}\to \A_{\rho'}$ obtained by restricting $X\in \A_{\rho}$ to the domain $D_{\rho'}$. When restricted to non-constant modes, its norm can be estimated as follows.

\begin{lemma}\label{lemma cutoff}
If $0<\phi \leq \e^{\rho-\rho'}$, then $\|\II (\Ii-\Ee)\| \leq \phi^{-1}$.
\end{lemma}

The proof is immediate.


\section{Renormalization}
\label{section:Renormalization}

\subsection{Renormalization operator}

We are only interested in vectors $\vecomega\in\Rr^d$ that satisfy $\vecomega\cdot\veck\not=0$ for any $\veck\in\Zz^d-\{0\}$.
Associated to $\sigma_{n-1}$ (specifying the cones $I_{n-1}^\pm$) and $\phi_n$ (the size of the analyticity strip width cut-off in \ref{subsection:Analyticity strip cut-off}, related to the domain we are considering), the $n$-th step renormalization operator is defined to be $\II_n \circ \LL_n \circ \UU_{n-1}$, where $\UU_{n-1}$ is the full elimination of the modes in $I^-_{n-1}$ as in Theorem \ref{thm unif} with $\vecv=\vecomega_{n-1}$ (for $t=1$). Vector fields in the domain of the renormalization operator for some pair $(\sigma_{n-1},\phi_n)$ are said to be renormalizable.

We denote the iteration of the renormalizations up to the step $n$ by
$$
\RR_n=\II_n \circ \LL_n \circ \UU_{n-1} \circ \RR_{n-1}
\quad\text{with}\quad
\RR_0=\id,
$$
which is defined on $\A'_\rho$. We highlight the fact that each $\RR_n$ is specified by a given set of pairs $\{(\sigma_{k-1},\phi_{k})\}_{1\leq k\leq n}$. 
A vector field inside the domain of $\RR_n$, for all $n$, is called infinitely renormalizable. Notice that $\RR_n(\vecomega_0+\vecv)=\vecomega_n$, for every $\vecv\in\Cc^d$ and any choice of a sequence pair $(\sigma,\phi)$. 


 Following the previous sections, the map $\RR_n$ is analytic on its domain.
Also, in case a vector field $X$ is real-analytic, the same is true for
$\RR_n(X)$.

\subsection{Coordinate changes}

Assuming that $X$ is in the domain of $\RR_n$, denote by $X_n=\RR_n(X)$ so that
\begin{equation}
X_n=\lambda_n\,(U_0 \circ T_1^{-1} \cdots U_{n-1}\circ T_n^{-1})^*(X)
\quad\text{and}\quad
X_0=X,
\end{equation}
where $U_n(X)=\fU_n(X_n)$. Thus,
\begin{equation}\label{formula X n with Vs}
P_n^*X_n= \lambda_n\, W_{n-1}(X)^* \cdots W_0(X)^*(X)
\end{equation}
with the isotopic to the identity diffeomorphisms
\begin{equation}
W_n(X)=
P_n^{-1}\circ U_n(X) \circ P_n.
\end{equation}
If $X$ is real-analytic, then $W_n(X)(\Rr^d)\subset\Rr^d$, since this property holds for $U_n(X)$. We also have $W_n(\Ii^+X)=\id$.

Finally, by Proposition \ref{properties rot}, $\rot\RR_n(X)=\lambda_n P_n\rot X$.


\subsection{Infinitely renormalizable vector fields}
\label{The Limit Set of the Renormalization}

We want to find sufficient conditions on sequences of pairs $(\sigma,\phi)$ in order to obtain infinitely renormalizable vector fields.

For any $\sigma\colon\Nn_0\to\Rr^+$ satisfying $\sigma_n<\|\vecomega_n\|$, we choose $\phi\colon\Nn\to(1,\infty)$ given by
\begin{equation}\label{choice sigma phi}
\phi_n  =\max\left\{1,
2 (d+1) \|\widetilde\LL_n\|
\frac{\varepsilon_{n-1}}{\varepsilon_n}
\right\}.
\end{equation}
Here $\varepsilon_n$ is as in \eqref{formula epsilon} for the $n$-th step.
Then, for any $\rho>0$, we associate the sequence of analyticity strip widths
\begin{equation}\label{formula rho n vf}
\rho_n = 
\frac{\rho - \BB_n(\vecomega,\sigma)}{A_0\dots A_{n-1}}, 
\quad
n\in\Nn,
\end{equation}
where
\begin{equation}\label{def B omega}
\BB_n(\vecomega,\sigma)=
\sum_{i=0}^{n-1} A_0\dots A_i \log\left(\e^{\delta+\nu/A_i}\phi_{i+1}\right) >0
\end{equation}
and $\nu$ and $\delta$ are the constants given in Theorem \ref{thm unif} and Lemma \ref{proposition TT}, respectively, taken to be the same for all $n$.
In case $\rho_n\leq0$ the renormalization procedure stops at step $n$.
Recall that each $A_n$ also depends on $\sigma_n$ as in \eqref{def An}.
Finally, define the function
\begin{equation}
\BB(\vecomega,\sigma)=\lim_{n\to+\infty}\BB_n(\vecomega,\sigma)
\end{equation}
whenever the limit exists (enough being bounded from above).

\begin{theorem}\label{convergence Rn}
If $X\in\A'_\rho$ is real-analytic and
\begin{itemize}
\item
$\rot X=\vecomega$, 
\item
$\|(\Ii-\Ee)X\|'_\rho < \varepsilon_0/(d+1)$,
\item
$\rho > \BB(\vecomega,\sigma)$,
\end{itemize}
then $X$ is infinitely renormalizable and
\begin{equation}\label{bound on LLn less than epsilon}
\|X_n-\vecomega_n\|'_{\rho_n}
< \varepsilon_n,
\qquad
n\in\Nn.
\end{equation}
\end{theorem}

\begin{remark}
Our choice \eqref{choice sigma phi} of $\phi$ is the ``smallest'' that we have achieved here so that $X$ is infinitely renormalizable. 
\end{remark}

\begin{proof}
Firstly we remark that by Proposition \ref{prop rot control EX},
\begin{equation}
\begin{split}
\|X-\vecomega\|'_\rho &\leq
\|X-\Ee X\|'_\rho + \|\Ee X-\vecomega\| \\
&\leq
 (d+1)\|(\Ii-\Ee)X\|'_\rho \\
& <\varepsilon_0.
\end{split}
\end{equation}
So, we can consider the bound on $\|X-\vecomega\|'_\rho$ instead of $\|(\Ii-\Ee)X\|'_\rho$.

Notice that if
\begin{equation}\label{cdn inf ren}
\|X_{n}-\vecomega_n\|'_{\rho_n} <\varepsilon_n
\end{equation}
(meaning that at each step $X_n$ is in the domain of $\UU_n$), then $X_n$ is renormalizable and $X_{n+1}=\RR_n(X)$. 
Being true for any $n\in\Nn$, then $X$ is infinitely renormalizable.
The inequality \eqref{cdn inf ren} can be estimated using Proposition \ref{prop rot control EX} and Lemmas \ref{lemma cutoff} and \ref{proposition TT} by
\begin{equation}
\begin{split}
\|X_{n}-\vecomega_n\|'_{\rho_n} 
&=
\|\II_n\LL_n\UU_{n-1}(X_{n-1})-\vecomega_n\|'_{\rho_n}
\\
&\leq
\|(\Ii-\Ee)\II_n\LL_n\UU_{n-1}(X_{n-1})\|'_{\rho_n} + \|\Ee\LL_n(X_{n-1})-\vecomega_n\| \\
&\leq
(d+1) \|\II_n(\Ii-\Ee)\LL_n(X_{n-1})\|'_{\rho_n}
\\ 
&\leq
(d+1)\frac{\|\widetilde\LL_n\|}{\phi_n} \|(\Ii-\Ee)\UU_{n-1}(X_{n-1})\|'_{\xi},
\end{split}
\end{equation}
where $\xi=A_{n-1}(\rho_n+\log\phi_n+\delta)$.

We now proceed by induction. Assuming that \eqref{bound on LLn less than epsilon} holds for $n-1$ and substituting the value of $\phi_n$,
\begin{equation}
\|X_{n}-\vecomega_n\|'_{\rho_n}
\leq
(d+1)\frac{2\|\widetilde\LL_n\|}{\phi_n} \|X_{n-1}-\vecomega_{n-1}\|'_{\xi'} <
\varepsilon_{n},
\end{equation}
where $\xi'=\xi+\nu=\rho_{n-1}$.
\end{proof}

\subsection{Width of resonance cones}

In the following sections we assume that the conditions of Theorem \ref{convergence Rn} are satisfied. In particular, we will be interested in those vectors $\vecomega$ for which we can find $\sigma$ as given in the lemma below so that $\BB(\vecomega,\sigma)$ converges and is less than $\rho$. 

Let 
$$
R_n=\frac1{2^n\|T_0\|\dots\|T_n\|}.
$$

\begin{lemma}\label{lemma cdns Rn}
If $\sigma\colon\Nn_0\to\Rr^+$ satisfies $\sigma_n\leq \|\vecomega_n\|$ and
\begin{equation}\label{cdn on sigma wrt Rn}
\lim_{n\to+\infty}\frac{\sigma_n\|P_n^{-1}\|}{R_n(R_{n-1}-R_n)\|\vecomega_n\|}=0,
\end{equation}
then there is $N\in\Nn$ such that, for $n\geq N$,
\begin{equation}\label{control on Rn ell for vf}
R_n \leq \frac{\rho_n}{\|P_n\|},
\qquad
\frac{42\|P_n^{-1}\|\, \varepsilon_n}{\sigma_{n}}
\leq \frac{R_{n-1}-R_n}{2\pi}
\end{equation}
and
\begin{equation}\label{estimates lim conv}
\lim_{n\to+\infty}
\frac{\|P_n^{-1}\|\,\varepsilon_n
}{R_n(R_{n-1}-R_n)\sigma_n} = 0,
\qquad
\lim_{n\to+\infty}
\frac{\|P_n^{-1}\|\,\varepsilon_n}{|\lambda_n|} = 0.
\end{equation}
\end{lemma}

\begin{proof}
For the first estimate in \eqref{control on Rn ell for vf} it is sufficient to check that $\inf_k\rho_k>0$. 
Indeed we get the bound for $N$ large enough such that $2^{-N}\leq\inf_k\rho_k$.
Now, as $\BB_n\leq \BB$ (notice the simplification of notation) we have that $\rho-\BB_n\geq\rho-\BB>0$. Hence we only need to check that $\sup_n A_0\dots A_{n-1}<\infty$. This follows from $\sum_{n\geq0} A_0\dots A_{n-1}< \delta^{-1}\BB<\infty$, so that $\lim_{n\to+\infty}A_0\dots A_{n-1}=0$.

 From \eqref{cdn on sigma wrt Rn} for any choice of $\kappa>0$ we can find $N\in\Nn$ so that for $n\geq N$,
$$
\sigma_n\leq \kappa \frac{(R_{n-1}-R_n)\|\vecomega_n\|}{\|P_n^{-1}\|}.
$$
Therefore, the second estimate in \eqref{control on Rn ell for vf} follows using the definition of $\varepsilon_n$, i.e. $\varepsilon_n\leq \sigma_n^2/\|\vecomega_n\|$.
This also proves the first limit in \eqref{estimates lim conv}.

Finally, notice that \eqref{eq def omega n} yields
$$
|\lambda_n|^{-1}\leq \|P_n\|\,\|\vecomega\| \,\|\vecomega_n\|^{-1} \leq \|T_0\|\dots\|T_n\|\,\|\vecomega\|\,\|\vecomega_n\|^{-1}.
$$
So,
$$
\frac{\|P_n^{-1}\|\,\varepsilon_n}{|\lambda_n|} \ll
\frac{\|P_n^{-1}\|\sigma_n^2}{R_n\|\vecomega_n\|^2}.
$$
We then obtain the second limit in \eqref{estimates lim conv} from \eqref{cdn on sigma wrt Rn}.
\end{proof}

\begin{remark}
It is simple to check that $\|g\circ P_n \|_{R_n} \leq \|g\|_{\rho_n}$.
\end{remark}


\section{Class of vectors}
\label{section:class vectors}

\begin{theorem}\label{thm class of vectors}
If $\vecomega\in\YY$, then there exists $\sigma\colon\Nn_0\to\Rr^+$ in the conditions of Lemma {\rm \ref{lemma cdns Rn}} such that $\BB(\vecomega,\sigma)<\infty$.
\end{theorem}

\begin{proof}
We need to estimate $\BB(\vecomega,\sigma)$. We start by noticing that
$$
\phi_{n+1}\ll
|\eta_{n+1}|\,\|T_{n+1}\| \, \frac{\varepsilon_n}{\varepsilon_{n+1}}
\ll
|\eta_{n+1}|\,\|T_{n+1}\| \, \|\vecomega_n\| \, \frac{\sigma_n^2}{\sigma_{n+1}^2}.
$$
From \eqref{cdn on sigma n for YY} we have $\sigma$ satisfying \eqref{cdn on sigma wrt Rn}. So,
$$
\BB(\vecomega,\sigma)\ll
\sum_{n\geq 0} 
A_0\dots A_n \log \left(|\eta_{n+1}|\,\|T_{n+1}\|\,\frac{\sigma_n\|\vecomega_{n+1}\|}{\sigma_{n+1}}\right)
$$
converges whenever $\vecomega\in\YY$ as in \eqref{defn YY vectors}.
\end{proof}

\section{Analytic conjugacy}
\label{section:Differentiable rigidity}

\subsection{$C^1$-conjugacy}
\label{subsection:from C0 to C1 conjugacy}

Starting from a vector field satisfying $\rot X=\vecomega$, we first use Theorem \ref{convergence Rn} to show that there is a smooth conjugacy.
Let $\Delta\subset\A'_\rho$ be the subset of all infinitely renormalizable real-analytic vector fields that verify the conditions of Theorem \ref{convergence Rn} while \eqref{cdn on sigma wrt Rn} holds.
Moreover, we fix $N\in\Nn$ as in Lemma \ref{lemma cdns Rn}.

We use the notation $\Diff_{per}^r$ for a set of isotopic to the identity $\Zz^d$-periodic $C^r$-diffeomorphisms, i.e. identify $\Diff_{per}^r(\Rr^d)$ with $\Diff_0^r(\Tt^d)$.

\begin{lemma}\label{lemma Wn-id}
For all $n\geq N$,
$W_n\colon \Delta\to\Diff_{per}^\omega(D_{R_n},D_{R_{n-1}})$ is analytic satisfying
\begin{equation}\label{norm of Wn-I in Drho n+1}
\|W_n(X)-\id\|_{R_n}
\leq 
\frac{42}{\sigma_{n}} \|P_n^{-1}\|\, \|X_{n}-\vecomega_n\|'_{\rho_n}.
\end{equation}
\end{lemma}

\begin{proof}
For any $X\in\Delta$, by \eqref{estimate U around X0} and the first inequality in \eqref{control on Rn ell for vf},
\begin{equation*}
\|W_n(X)-\id\|_{R_n} =
\|P_n^{-1}(U_n(X) -\id) P_n\|_{R_n} 
\leq
\frac{42}{\sigma_{n}} \|P_n^{-1}\|\, \|X_{n}-\vecomega_n\|'_{\rho_n}.
\end{equation*}

Now, for $\vecx\in D_{R_n}$,
\begin{equation*}
\begin{split}
\|\im W_n(X)(\vecx)\| 
& \leq 
\|\im(W_n(X)(\vecx)-\vecx)\| + \|\im \vecx\|
\\
&<
\|W_n(X)-\id\|_{R_n}+R_n/2\pi
\leq
R_{n-1}/2\pi,
\end{split}
\end{equation*}
where we have used the second inequality in \eqref{control on Rn ell for vf}. Hence $W_n(X)\colon D_{R_n}\to D_{R_{n-1}}$.
 From the properties of $\fU_n$, $W_n\colon\Delta\to \Diff_{per}^\omega(D_{R_n},D_{R_{n-1}})$ is analytic.
\end{proof}

For $n\geq m\geq 0$ consider the analytic map $H_{m,n}\colon \Delta\to\Diff_{per}^\omega(D_{R_n},\Cc^d)$ defined by
\begin{equation}\label{def H n}
H_{m,n}(X)=W_m(X)\circ\dots\circ W_n(X).
\end{equation}
In particular, by the above lemma, $H_{m,n}(X)\colon D_{R_n} \to D_{R_{m-1}}$ whenever $m\geq N$.

\begin{lemma}\label{lemma Hn-Hn-1 vf}
For $X\in\Delta$ and $n>m\geq N$,
\begin{equation}\label{estimate Hn-Hn-1}
\begin{split}
\|H_{m,n}(X)-H_{m,n-1}(X)\|_{R_n}
&\leq
\frac{R_{m-1}}{\frac{(R_{n-1}-R_n)\sigma_n}{2\cdot42\|P_n^{-1}\|\,\|X_n-\vecomega_n\|'_{\rho_n}}-1}
\\
\|H_{m,n}(X)-\id\|_{R_n} 
&\leq
42\sum_{i=m}^n\frac{\|P_i^{-1}\|}{\sigma_i} \|X_i-\vecomega_i\|'_{\rho_i}.
\end{split}
\end{equation}
\end{lemma}

\begin{proof}
For each $k=m,\dots, n-1$, consider the transformations
\begin{equation*}
\begin{split}
G_k(z,X)= &
 (W_k(X)-\id)\circ(\id+G_{k+1}(z,X))+G_{k+1}(z,X),
\\
G_n(z,X)= &
z(W_n(X)-\id),
\end{split}
\end{equation*}
with $(z,X)\in  \{z\in\Cc\colon |z| < 1+ d_n \} \times \Delta$, where
$$
d_n=\frac{R_{n-1}-R_n}{2\pi}\frac{\sigma_n}{42\|P_n^{-1}\|\,\|X_n-\vecomega_n\|'_{\rho_n}}-1
$$
is positive by the second inequality in \eqref{control on Rn ell for vf}.
If the image of $D_{R_n}$ under $\id+G_{k+1}(z,X)$ is inside the domain of $W_k(X)$, or simply
$$
\|G_{k+1}(z,X)\|_{R_n}\leq (R_k-R_n)/2\pi,
$$
then $G_k$ is well-defined as an analytic map into $\Diff_{per}^\omega(D_{R_n},\Cc^d)$, and
\begin{equation}\label{bnd Gk}
\|G_k(z,X)\|_{R_n} \leq 
\|W_k(X)-\id\|_{R_k} + \|G_{k+1}(z,X)\|_{R_n}.
\end{equation}
An inductive scheme shows that 
\begin{equation*}
\begin{split}
\|G_n(z,X)\|_{R_n} 
= &
|z|\,\|W_n(X)-\id\|_{R_n}
\\
\leq &
(R_{n-1}-R_n)/2\pi,
\\
\|G_k(z,X)\|_{R_n} 
\leq &
\sum_{i=k}^{n-1}
\|W_i(X)-\id\|_{R_i}
+
\|G_n(z,X)\|_{R_n}
\\
\leq &
(R_{k-1}-R_{n})/2\pi,
\end{split}
\end{equation*}
using Lemmas \ref{lemma Wn-id} and \ref{lemma cdns Rn}.

It is easy to check that $G_m(1,X)=H_{m,n}(X)-\id$ and $G_m(0,X)=H_{m,n-1}(X)-\id$. In particular, from \eqref{bnd Gk} and Lemma \ref{lemma Wn-id}
\begin{equation}
\begin{split}
\|H_{m,n}(X)-\id\|_{R_n} 
&=
\|G_m(1,X)\|_{R_n}
\leq 
\sum_{i=m}^n\|W_i(X)-\id\|_{R_i} \\
&\leq
\sum_{i=m}^n\frac{42 \|P_i^{-1}\|}{\sigma_i} \|X_i-\vecomega_i\|'_{\rho_i}.
\end{split}
\end{equation}
Finally, by Cauchy's formula
\begin{equation*}
\begin{split}
\|H_{m,n}(X)-H_{m,n-1}(X)\|_{R_{n}}
&=
\|G_m(1,X)-G_m(0,X)\|_{R_{n}} \\
&=
\norm{\frac1{2\pi i} \oint_{|z|=1+d_n/2}\frac{G_m(z,X)}{z(z-1)}dz}_{R_{n}} \\
&\leq 
\frac2{d_n}\sup_{|z|=1+d_n/2}\|G_m(z,X)\|_{R_{n}} 
\leq
\frac{R_{m-1}}{\pi d_n}.
\end{split}
\end{equation*}
\end{proof}

\begin{lemma}\label{lemma conv Hmn}
There exists an analytic map $H\colon\Delta\to \Diff^1_{0}(\Tt^d)$ such that
\begin{equation}
H(X)=\lim_{n\to+\infty}H_{0,n}(X).
\end{equation}
\end{lemma}

\begin{proof}
Consider $C^1_{per}(\Rr^d,\Cc^d)$ to be the Banach space of the $\Zz^d$-periodic $C^1$ functions between $\Rr^d$ and $\Cc^d$ with norm
\begin{equation}
\|f\|_{C^1} = \max_{k\leq1}\max_{\vecx\in\Rr^d} \|D^kf(\vecx)\|.
\end{equation}

Let $m\geq N$.
Since the domains $D_{R_n}$ are shrinking, consider the restrictions of $W_n(X)$ and $H_{m,n}(X)$ to $\Rr^d$. So, for any $X\in\Delta$, Cauchy's estimate and Lemma \ref{lemma Hn-Hn-1 vf} holds that
\begin{equation}\label{estimate var H n}
\begin{split}
\|H_{m,n}(X)-H_{m,n-1}(X)\|_{C^1} 
&\leq 
\max_{k\leq1}\sup_{\vecx\in D_{R_n/2}} \|D^k [H_{m,n}(X)(\vecx)-H_{m,n-1}(X)(\vecx)]\| 
\\
&\leq C_1
\frac{1}{R_n}  \|H_{m,n}(X)-H_{m,n-1}(X)\|_{R_n}
\end{split}
\end{equation}
goes to zero by \eqref{estimate Hn-Hn-1} and the first bound in \eqref{estimates lim conv}, where $C_1>0$ is a constant. 
Thus, $H_{m,n}(X)$ converges to $H_m(X)$ in $C_{per}^1(\Rr^d,\Cc^d)$.
Similarly, there is a constant $C_2>0$ such that $\|H_m(X)-\id\|_{C^1} \leq C_2 \sigma_m^{-1}\|P_m^{-1}\|\,\|X_m-\vecomega_m\|'_{\rho_m}$. For $m$ large enough we can use the first limit in \eqref{estimates lim conv} to show that $\|H_m(X)-\id\|_{C^1}<1$. Hence, $H_m(X)$ is a diffeomorphism isotopic to the identity.

The convergence of $H_{m,n}$ is uniform in $\Delta$ so $H_m$ is an analytic map. The fact that, for real-analytic $X$, $H_m(X)$ takes real values for real arguments, follows from the same property for each $W_n(X)$.

Take now the analytic map 
$$
H=W_0\dots W_{m-1}\circ H_m.
$$
It follows from the above that for every real-analytic $X\in\Delta$, $H(X)\in\Diff_0^1(\Tt^d)$.
\end{proof}

\begin{theorem}\label{lemma conj}
For every real-analytic $X\in \Delta$, $H(X)^*(X)= \vecomega$ on $\Rr^d$.
\end{theorem}

\begin{proof}
 From \eqref{formula X n with Vs} we have 
$$
H_{0,n}(X)^*(X) -\vecomega=
\lambda_n^{-1} P_n^* (X_n)-\vecomega= 
\lambda_n^{-1} P_n^* (X_n-\vecomega_n).
$$
Since $\|\lambda_n^{-1} P_n^* (X_n-\vecomega_n)\|_{C^0} \leq |\lambda_n^{-1}|\,\|P_n^{-1}\|\,\|X_n-\vecomega_n\|'_{\rho_n} \to 0$ by the second limit in \eqref{estimates lim conv} and $H_{0,n}\to H$ as $n\to+\infty$, we complete the proof.
\end{proof}

\subsection{From $C^1$ to $C^\omega$-conjugacy}

Because of the analyticity dependence of the conjugacy map $H$ with respect to the vector field $X$, we will show that the conjugacy can be extended analytically to a complex strip.

\begin{lemma}\label{proposition: comm H and R}
If $X\in\Delta$ and $\vecx\in\Rr^d$, then
\begin{equation}\label{eq comm rel for H}
H(X\circ R_\vecx)= R_\vecx^{-1}\circ H(X)\circ R_\vecx.
\end{equation}
\end{lemma}

\begin{proof}
The relations \eqref{relation tilde U and U} 
and \eqref{commutation property Rz and Ln}
yield that $\UU_n(X\circ R_\vecx)=\UU_n(X)\circ R_\vecx$ and
$\LL_{n}(X\circ R_\vecx)=\LL_{n}(X)\circ R_{T_n\vecx}$.
This implies immediately that
\begin{equation}\label{relation Rn and Rz}
\RR_n(X\circ R_\vecx)=\RR_n(X)\circ R_{{P_n} \vecx}.
\end{equation}
Next, from a simple adaptation of \eqref{relation tilde U and U} and the formula $R_{P_n\vecz}=P_n R_\vecz {P_n}^{-1}$ for $\vecz\in\Cc^d$, we get
\begin{equation}
\begin{split}
W_n(X\circ R_\vecx) =&
P_n^{-1} \circ\fU_{ a_n}(\LL_n\RR_{n-1}(X\circ R_\vecx))\circ P_n
\\
=&
R_\vecx^{-1}\circ W_n(X)\circ R_\vecx.
\end{split}
\end{equation}
Thus, $H_{0,n}(X\circ R_\vecx)= R_\vecx^{-1}\circ H_{0,n}(X)\circ R_\vecx$.
The convergence of $H_{0,n}$ implies \eqref{eq comm rel for H}.
\end{proof}


\begin{theorem}\label{thm C1 to analytic}
For every real-analytic $X\in\Delta$, $H(X)\in\Diff_0^\omega(\Tt^d)$.
\end{theorem}

\begin{proof}
Take $r$ such that $\rho>r>\BB(\vecomega,\sigma)$. 
For each $X\in\Delta$ and $\vecz\in D_\eta$ with $\eta=\rho-r>0$, $X\circ R_\vecz\in \A'_r$ and $\|X\circ R_\vecz-\vecomega\|_r\leq \|X-\vecomega\|_\rho$. So, we define the set of infinitely renormalizable vector fields $\Delta_r\subset\A'_r$ and can use the previous results to show the existence of the conjugacy $H(X\circ R_\vecz)$.

We now need to analytically extend the conjugacy $H(X)$ to a complex neighbourhood of $\Rr^d$.
Let $F(\vecz)=\vecz+ H(X\circ R_\vecz)(0)$.
The maps $\vecz\mapsto X\circ R_\vecz$, $X\mapsto H(X)$ and $C^1_{per}(\Rr^d,\Cc^d)\ni g\mapsto g(0)$ are analytic, thus $F$ is also analytic on $D_\eta$ and $F-\id$ is $\Zz^d$-periodic.
It remains to show that $F$ is an analytic extension of $H(X)$.
 From \eqref{eq comm rel for H}, for
any $\vecx\in\Rr^d$, we have
\begin{equation}
\begin{split}
F(\vecx)
& = 
\vecx+R_\vecx^{-1}\circ H(X)\circ R_\vecx(0) \\
& = 
\vecx+H(X)(\vecx)-\vecx \\
& = 
H(X)(\vecx).
\end{split}
\end{equation}
\end{proof}


\section{Small analyticity strip}\label{section:proof local rig thm}

The above results (viz. Theorems \ref{convergence Rn} and \ref{thm C1 to analytic}) can be generalized for a small analyticity radius $\rho$. Thus, through Theorem \ref{thm class of vectors}, we conclude the proof of Theorem \ref{main thm I}.

\begin{theorem}\label{thm small analyt radius}
Let $\vecomega\in\YY$.
If $X\in\A_r$, $r>0$, $\rot X=\vecomega$ and $X$ is sufficiently close to constant, then there is an analytic diffeomorphism $\psi$ such that $\widetilde X=\psi^*X$ is in the conditions of Theorem {\rm \ref{convergence Rn}}, thus infinitely renormalizable.
\end{theorem}

\begin{proof}
First, we observe that $X\in\A'_\rho$ for some $\rho<r$ because $\|X\|'_\rho\leq(1+2\pi/(r-\rho))\|X\|_r$.
Moreover, by Proposition \ref{prop rot control EX},
$$
\|X-\vecomega\|'_\rho  \leq
(d+1) \|(\Ii-\Ee)X\|'_\rho.
$$

By considering a sufficiently large $N$ we want to apply Theorem \ref{convergence Rn} to 
$$
\widetilde X=\LL_N\UU_{N-1}\dots\LL_1\UU_0(X)
$$ 
in $\A'_{\rho_N}$, with $\|(\Ii-\Ee)X\|'_\rho$ small enough. (Notice that we are not including any operator $\II$, thus no need for $\phi$.)
Under a suitable choice of resonance width and stopping-time sequences up to step $N$, we recover the large strip case since 
\begin{equation}\label{formula rho N}
\rho_N=
\frac{\rho-\sum_{i=0}^{N-1}A_0\cdots A_i(\delta+\nu/A_i)}{A_0\cdots A_{N-1}}.
\end{equation}
The numerator above can be made larger than some positive constant for any choice of a finite $N$.
It remains to check that $\rho_N>\BB(\vecomega_N,\sigma')$, now for $\sigma'_n=\sigma_{N+n}$ given by Theorem \ref{thm class of vectors}.
As
\begin{equation}\label{est BB N}
\BB(\vecomega_N,\sigma') =
\frac{(\BB-\BB_N)(\vecomega,\sigma)}{A_0\dots A_{N-1}}
\end{equation}
we compare \eqref{est BB N} with \eqref{formula rho N} by noticing that $\BB_N\to\BB$ as $N\to+\infty$.
So, for $N$ sufficiently large (but finite), \eqref{est BB N} can be made less than $\rho_N$.
\end{proof}

\begin{theorem}\label{lemma: existence of h diff}
Let $\vecomega\in\YY$.
If $v\in\vf^\omega(\Tt^d)$ generates a flow with rotation vector $\vecomega$ and $v$ is sufficiently close to constant, then there exists $h\in\Diff_0^\omega(\Tt^d)$ such that 
\begin{equation}\label{equivalence vf and omega}
h^*(v)=\vecomega.
\end{equation}
\end{theorem}

\begin{proof}
The lift to $\Rr^d$ of $v$ is assumed to have an analytic extension in $D_r$. Theorem \ref{thm small analyt radius} then gives $\widetilde v\in\A_\rho$ and, as long as
$v$ is close enough to constant, we have $\widetilde v\in \Delta$.
Then, by theorems \ref{thm C1 to analytic} and \ref{lemma conj}, the analytic diffeomorphism $h=H(\widetilde v)\bmod1$ verifies \eqref{equivalence vf and omega}.
\end{proof}


\section*{Acknowledgements} 

I would like to express my gratitude to Kostya Khanin and Jens Marklof for fruitful discussions.
The author was partially supported by Funda\c c\~ao para a Ci\^encia e a Tecnologia through the Program POCI 2010.


\bibliographystyle{plain} 
\bibliography{rfrncs}

\end{document}